\documentclass{amsart}

\usepackage{amsmath} 
\usepackage{amssymb}

\newcommand{\forces}{\Vdash} 
\newcommand{\con}{{\mathfrak c}}

\newcommand{\can}{{}^{\omega}2} 
\newcommand{\baire}{{}^{\omega}\omega}
 
\newcommand{\rest}{{\restriction}}
\newcommand{\dom}{{\rm dom}}

\newcommand{\pfs}{{}^{\mathunderaccent\smile-3 \omega}2}

\newcount\skewfactor
\def\mathunderaccent#1#2 {\let\theaccent#1\skewfactor#2
\mathpalette\putaccentunder}
\def\putaccentunder#1#2{\oalign{$#1#2$\crcr\hidewidth
\vbox to.2ex{\hbox{$#1\skew\skewfactor\theaccent{}$}\vss}\hidewidth}}
\def\name{\mathunderaccent\tilde-3 }

\newcommand{\cA}{{\mathcal A}}

\newcommand{\bE}{{\bf E}}
\newcommand{\bF}{{\bf F}}

\newcommand{\cI}{{\mathcal I}}
\newcommand{\cJ}{{\mathcal J}}

\newcommand{\cN}{{\mathcal N}}
\newcommand{\cM}{{\mathcal M}}

\newcommand{\bbP}{{\mathbb P}}
\newcommand{\bbQ}{{\mathbb Q}}

\newcommand{\mbR}{{\mathbb R}}

\newcommand{\bbS}{{\mathbb S}}

\newcommand{\bV}{{\mathbf V}}

\newcommand{\cZ}{{\mathcal Z}}
\newcommand{\bbZ}{{\mathbb Z}}

\newtheorem{theorem}{Theorem}[section] 
\newtheorem{claim}{Claim}[theorem]
\newtheorem{lemma}[theorem]{Lemma} 
\newtheorem{proposition}[theorem]{Proposition} 
\newtheorem{corollary}[theorem]{Corollary} 
\newtheorem{observation}[theorem]{Observation} 

\theoremstyle{definition}
\newtheorem{problem}[theorem]{Problem} 
 
\newtheorem{definition}[theorem]{Definition}

\theoremstyle{remark}

\newtheorem{remark}[theorem]{Remark}

\begin{document}

\title{Small--large subgroups of the reals}

\author{Andrzej Ros{\l}anowski}
\address{Department of Mathematics\\
 University of Nebraska at Omaha\\
 Omaha, NE 68182-0243, USA}
\email{roslanow@member.ams.org}
\urladdr{http://www.unomaha.edu/logic}

\author{Saharon Shelah}
\address{Institute of Mathematics\\
 The Hebrew University of Jerusalem\\
 91904 Jerusalem, Israel\\
 and  Department of Mathematics\\
 Rutgers University\\
 New Brunswick, NJ 08854, USA}
\email{shelah@math.huji.ac.il}
\urladdr{http://www.math.rutgers.edu/$\sim$shelah}

\thanks{Both authors acknowledge support from the United States-Israel
  Binational Science Foundation (Grant no. 2010405). Publication 1081 of the
  second author.}

\date{May 23, 2016}

\begin{abstract}
  We are interested in subgroups of the reals that are small in one
  and large in another sense. We prove that, in ZFC, there exists a
  non--meager Lebesgue null subgroup of $\mbR$, while it is consistent
  that there there is no non--null meager subgroup of $\mbR$. This
  answers a question from Filipczak, Ros{\l}anowski and Shelah
  \cite{FRSh:1031}.  
\end{abstract}
\maketitle

\section{Introduction}
Subgroups of the reals which are small in one and large in another sense
were crucial in Filipczak, Ros{\l}anowski and Shelah \cite{FRSh:1031}. If
there is a non--meager Lebesgue null subgroup of $(\mbR,+)$, then there is
no translation invariant Borel hull operation on the $\sigma$--ideal $\cN$
of Lebesgue null sets. That is, there is no mapping $\psi$ from $\cN$ to
Borel sets such that for each null set $A\subseteq \mbR$:
\begin{itemize}
\item $A\subseteq \psi(A)$ and $\psi(A)$ is null, and
\item $\psi(A+t)=\psi(A)+t$ for every $t\in \mbR$. 
\end{itemize}
Parallel claims hold true if ``Lebesgue null'' is interchanged with
``meager'' and/or $(\mbR,+)$ is replaced with $(\can,+_2)$.  

If $\cM$ is the $\sigma$--ideal of meager subsets of $\mbR$ (and $\cN$ is
the null ideal on $\mbR$) and $\{\cI,\cJ\}=\{\cN,\cM\}$, then various set
theoretic assumptions imply the existence of a subgroup of $\mbR$ which
belongs to $\cI$ but not to $\cJ$. But in \cite[Problem 4.1]{FRSh:1031} we
asked if the existence of such subgroups can be shown in ZFC. This question
is interesting {\em per se}, regardless of its connections to translation
invariant Borel hulls.

The present paper presents two theorems. First, in Theorem
\ref{firstres} we give ZFC examples of null non-meager subgroups of
$(\can,+_2)$ and $(\mbR,+)$, respectively. Next in Theorem
\ref{mainthm} we show that it is consistent with ZFC that every meager
subgroup of $(\can,+_2)$ and/or $(\mbR,+)$ has Lebesgue measure
zero. This answers \cite[Problem 4.1]{FRSh:1031}. Also, our results
give another example of a strange asymmetry between measure and
category.  \bigskip

\noindent {\bf Notation}\qquad Our notation is rather standard and
compatible with that of classical textbooks (like Jech \cite{J} or
Bartoszy\'nski and Judah \cite{BaJu95}). However, in forcing we keep the
older convention that {\em a stronger condition is the larger one}. 

\begin{enumerate}

\item The Cantor space $\can$ of all infinite sequences with values 0 and 1
is equipped with the natural product topology, the product measure $\lambda$
and the group operation of coordinate-wise addition $+_2$ modulo 2.  
\item Ordinal numbers will be denoted be the lower case initial letters of
  the Greek alphabet $\alpha,\beta,\gamma,\delta$. Finite ordinals
  (non-negative integers) will be denoted by letters $i,j,k,\ell,m,n$ while
  integers will be called $L,M$.
\item Most of our intervals will be intervals of non-negative integers, so
  $[m,n)=\{k\in \omega: m\leq k<n\}$ etc. They will be denoted by letter $J$
  (with possible indices). However, we will also use the notation $[0,1)$ to
  denote the unit interval of reals.
\item The Greek letter $\kappa$ will stand for an uncountable cardinal such
  that $\kappa^{\aleph_0} = \kappa\geq\aleph_2$.
\item For a forcing notion $\bbP$, all $\bbP$--names for objects in
  the extension via $\bbP$ will be denoted with a tilde below (e.g.,
  $\name{\tau}$, $\name{X}$), and $\name{G}_\bbP$ will stand for the
  canonical $\bbP$--name for the generic filter in $\bbP$.
\item We fix a well ordering $\prec^*$ of all hereditarily finite sets.  
\item The set of all partial finite functions with domains included in
  $\omega$ and with values in $2$ is denoted  $\pfs$.  
\end{enumerate}

\section{Null non--meager}
Here we will give a ZFC construction of a non--meager Lebesgue null subgroup
of the reals. The main construction is done in $\can$ and then we transfer
it to $\mbR$ using the standard binary expansion $\bE$.  

\begin{definition}
\label{binexp}
Let $D_0^\infty=\{x\in \can:(\exists^\infty i<\omega)(x(i)=0)\}$ and for
$x\in D^\infty_0$ let $\bE(x)=\sum\limits_{i=0}^\infty x(i)2^{-(i+1)}$. 
\end{definition}

\begin{proposition}
\label{expanding}
\begin{enumerate}
\item The function $\bE:D^\infty_0\longrightarrow [0,1)$ is a continuous
bijection, it preserves both the measure and the category.   
\item Assume that 
\begin{enumerate}
\item[(a)] $x,y,z\in D^\infty_0$, $\bE(z)=\bE(x)+\bE(y)$ modulo 1, and 
\item[(b)] $n<m<\omega$ and both $x\rest [n,m]$ and $y\rest [n,m]$ are
  constant. 
\end{enumerate}
Then $z\rest [n,m-1]$ is constant. 
\item Assume that 
\begin{enumerate}
\item[(a)] $x,y\in D^\infty_0$, $0<\bE(x)$ and $\bE(y)=1-\bE(x)$, 
\item[(b)] $n<m<\omega$ and $x\rest [n,m]$ is constant. 
\end{enumerate}
Then $y\rest [n,m-1]$ is constant. 
\end{enumerate}
\end{proposition}

\begin{proof}
(1)\qquad Well known.

\noindent (2,3)\quad Straightforward (just consider the possible constant
values and analyze how the addition is performed).
\end{proof}

\begin{theorem}
\label{firstres}
\begin{enumerate}
\item There exists a null non-meager subgroup of $(\can,+_2)$.
\item There exists a null non-meager subgroup of $(\mbR,+)$.
\end{enumerate}
\end{theorem}

\begin{proof}
(1)\qquad For $k\in\omega$ let $n_k=\frac{1}{2}k(k+1)$ and let $D$ be a
non-principal  ultrafilter on $\omega$. Define  
\[H_D=\Big\{x\in \can:\big(\exists m<\omega\big)\big(\exists j<2\big)
\big(\big\{k>m: x\rest [n_k,n_{k+1}-m)\equiv j\big\}\in D \big)
\Big\}.\]   
\begin{enumerate}
\item[(i)] $H_D$ is a subgroup of $(\can,+_2)$.
\end{enumerate}
Why? Suppose that $x_0,x_1\in H_D$ and let $m_\ell<\omega$ and $j_\ell<2$ be 
such that 
\[A_\ell\stackrel{\rm def}{=}\big\{k>m_\ell: x_\ell\rest [n_k,n_{k+1} 
-m_\ell)\equiv j_\ell\big\}\in D.\] 
Let $m=\max(m_0,m_1)$ and $j=j_0-_2 j_1$. Then  $A_0\cap A_1\in D$ and for
each $k\in A_0\cap A_1$ we have $(x_0-_2 x_2)\rest [n_k,n_{k+1}-m)\equiv
j$. Hence $x_0-_2 x_1\in H_D$.   

\begin{enumerate}
\item[(ii)] $H_D\in\cN$.
\end{enumerate}  
Why? For each $m<k<\omega$ and $j<2$ we have
\[\lambda(\{x\in\can: x\rest [n_k,n_{k+1}-m)\equiv j\})=2^{m-(k+1)}\]  
and therefore for each $m<\omega$ and $j<2$
\[\lambda(\{x\in\can: (\exists ^\infty k)(x\rest [n_k,n_{k+1}-m)\equiv 
j)\})=0.\]  
Now note that $H_D\subseteq \bigcup\limits_{m<\omega}\bigcup\limits_{j<2}
\big\{x\in\can: (\exists ^\infty  k)(x\rest [n_k,n_{k+1}-m)\equiv
j)\big\}$. 

\begin{enumerate}
\item[(iii)] $H_D\notin \cM$.
\end{enumerate}
Why? Suppose that $W$ is a dense $\Pi^0_2$ subset of $\can$. Then we may choose
an increasing sequence $\langle k_i:i\in\omega\rangle$ and a function $f\in \can$ 
such that 
\[\Big\{x\in\can: \big(\exists^\infty i\big)\big(x\rest [n_{k_i},n_{k_{i+1}})=
f\rest [n_{k_i},n_{k_{i+1}})\big)\Big\}\subseteq W.\]
Let $A=\bigcup \{[k_{2i},k_{2i+1}):i\in\omega\}$ and 
$B=\bigcup \{[k_{2i+1},k_{2i+2}):i\in\omega\}$. Then either $A\in D$ or $B\in D$. 
Let $x_A,x_B\in\can$ be such that, for each $i\in \omega$,
\[\begin{array}{l}
x_A\rest [n_{k_{2i}},n_{k_{2i+1}})\equiv 0,\quad 
x_A\rest [n_{k_{2i+1}},n_{k_{2i+2}})=f\rest n_{k_{2i+1}},n_{k_{2i+2}})\quad 
\mbox{ and}\\
x_B\rest [n_{k_{2i+1}},n_{k_{2i+2}})\equiv 0,\quad 
x_B\rest [n_{k_{2i}},n_{k_{2i+1}})=f\rest n_{k_{2i}},n_{k_{2i+1}}).
\end{array}\]
Then $x_A,x_B\in W$ and either $x_A\in H_D$ or $x_B\in H_D$. Consequently, $W\cap 
H_D\neq\emptyset$.
\medskip

\noindent (2)\qquad Consider $H_D^*=\bE[H_D\cap D^\infty_0]+\bbZ$. It
follows from \ref{expanding}(1) that $H_D^*$ is a Lebesgue null meager
subset of $\mbR$. We will show that it is a subgroup of $(\mbR,+)$.

Suppose that $x_0,x_1\in H_D\cap D^\infty_0$ and $L_0,L_1\in\bbZ$ and we
will argue that $(\bE(x_0)+L_0)+(\bE(x_1)+L_1)\in H_D^*$. Let $m_\ell<\omega$
be  such that   
\[A_\ell\stackrel{\rm def}{=}\big\{k>m_\ell: x_\ell\rest [n_k,n_{k+1} 
-m_\ell)\mbox{ is constant }\big\}\in D\]  
and let $m=\max(m_0,m_1)+1$. Choose $y\in D^\infty_0$ and $M\in \{0,1\}$
such that $\bE(x_0)+\bE(x_1)=\bE(y)+M$. It follows from \ref{expanding}(2) that
for every $k\in A_0\cap A_1$, $k>m$, we have that $y\rest [n_k,n_{k+1} -m)$
is constant and since $A_0\cap A_1\in D$ we conclude $y\in
H_D$. Consequently, $(\bE(x_0)+L_0)+(\bE(x_1)+L_1)=\bE(y)+(M+L_0+L_1)\in
H_D^*$. 

Now assume that $x\in H_D\cap D^\infty_0$, $L\in\bbZ$ and we
will argue that $-(\bE(x)+L)\in H_D^*$. If $\bE(x)=0$ then the assertion is 
clear, so assume also $\bE(x)>0$. Let $m<\omega$ be  such that    
\[A\stackrel{\rm def}{=}\big\{k>m: x\rest [n_k,n_{k+1} -m)\mbox{ is constant
} \big\}\in D.\]  
Choose $y\in D^\infty_0$ such that $1-\bE(x)=\bE(y)$. It follows from
\ref{expanding}(3) that for every $k\in A$, $k>m+1$, we have that $y\rest
[n_k,n_{k+1} -(m+1))$ is constant. Consequently, $y\in H_D$ and 
$-(\bE(x)+L)=\bE(y)-1-L\in H_D^*$. 
\end{proof}

\begin{remark}
A somewhat simpler non--meager null subgroup of $(\can,+_2)$ is
\[H_D^-=\Big\{x\in \can:\big\{k\in\omega: x\rest [n_k,n_{k+1})\equiv 0\big\}
\in D\Big\}.\]  
The group $H_D$, however, was necessary for our construction of
$H^*_D<\mbR$.   
\end{remark}

\begin{corollary}
There exists no translation invariant Borel hull for the null ideal on
$\can$ and/or on $\mbR$.
\end{corollary}

\section{Some technicalities}
Here we prepare the ground for our consistency results. 

\subsection{Moving from $\mbR$ to $\can$}
First, let us remind connections between the addition in $\mbR$ and
that of $\can$ (via the binary expansion $\bE$, see \ref{binexp}). 

\begin{definition}
  Let $J=[m,n)$ be a non-empty interval of integers and $c\in\{0,1\}$. For
  sequences $\rho,\sigma \in {}^J2$ we define $\rho\circledast_c \sigma$ as
  the unique $\eta\in {}^J 2$ such that
\[\Big(\sum_{i=m}^{n-1}\rho(i)2^{-(i+1)}+\sum_{i=m}^{n-1}
\sigma(i)2^{-(i+1)} +c\cdot 2^{-n}\Big) - \sum_{i=m}^{n-1} \eta(i)2^{-(i+1)}
\in \{0,2^{-m}\}.\]
For notational convenience we also set $\rho\circledast_2\sigma= \rho +_2
\sigma$ (coordinate-wise addition modulo 2). 
\end{definition}

The operation $\circledast_c$ is defined on the set ${}^J2$, so it does
depend on $J$. We may, however, abuse notation and use that same
symbol $\circledast_c$ for various $J$. 

\begin{observation}
\label{carrying}
Let $m,\ell,n$ be integers such that $m<\ell<n$ and let $J=[m,n)$.
\begin{enumerate}
\item For each $c\in \{0,2\}$, $({}^J2,\circledast_c)$ is an Abelian group. 
\item If $\rho,\sigma\in {}^J 2$ and $\rho(\ell)=\sigma(\ell)$,
  then $(\rho\circledast_0\sigma)\rest [m,\ell)=
  (\rho\circledast_1\sigma)\rest [m,\ell)$.
\item If $\rho,\sigma\in {}^J 2$ and $(\rho \circledast_0
  \sigma)(\ell)=0$, then $(\rho\circledast_0\sigma)\rest [m,\ell)=  
  (\rho\circledast_1\sigma)\rest [m,\ell)$.
\item Suppose that  $r,s\in [0,1)$, $\rho,\sigma,\eta\in D^\infty_0$,
  $\bE(\rho)=r$, $\bE(\sigma)=s$ and $\bE(\eta)=r+s$ modulo 1. Then  
  \begin{itemize}
  \item if $\sum\limits_{i\geq n} \big((\rho(i)+\sigma(i))/2^{i+1}
    \big)\geq 2^{-n}$, then $\eta\rest J=  (\rho\rest J)\circledast_1
    (\sigma\rest J)$;  
  \item if $\sum\limits_{i\geq n} \big((\rho(i)+\sigma(i))/2^{i+1}
    \big)< 2^{-n}$, then $\eta\rest J= (\rho\rest J)\circledast_0
    (\sigma\rest J)$.  
  \end{itemize}
\end{enumerate}
\end{observation}

\subsection{The combinatorial heart of our forcing arguments} 
For this subsection we fix a strictly increasing sequence $\bar{n}=\langle
n_j:j<\omega\rangle\subseteq \omega$. 

\begin{definition}
\label{translation}
We define $\bar{m}[\bar{n}]=\langle
m_i:i<\omega\rangle$, $\bar{N}[\bar{n}]=\langle N(i):i<\omega\rangle$, 
$\bar{J}[\bar{n}]=\langle J_i:i<\omega\rangle$, $\bar{H}[\bar{n}]=\langle
H_i: i<\omega\rangle$, $\pi[\bar{n}]=\langle \pi_i:i<\omega\rangle$ and
$\bF[\bar{n}]$ as follows. 

We set $m_0=0$ and then inductively for  $i<\omega$ we let 
\begin{enumerate}
\item[$(*)_1$] $m_{i+1}=2^{n_{m_i}+1081}$. 
\end{enumerate}
Next, for $i<\omega$,  
\begin{enumerate}
\item[$(*)_2$] $N(i)=n_{m_i}$, $J_i=\big[N(2^i),N(2^{i+1})\big)$, and 
\item[$(*)_3$] $H_i=\big\{a\subseteq {}^{J_i}2:(1-2^{-N(2^i)})\cdot 2^{|J_i|}\leq |a|
  \big\}$.
\end{enumerate}
We also set  $\pi_i:|H_i|\longrightarrow H_i$ to be the $\prec^*$--first
bijection from $|H_i|$ onto $H_i$. Finally, for
$\eta\in\prod\limits_{m<\omega} (m+1)$ we let  
\begin{enumerate}
\item[$(*)_4$] $\bF_0[\bar{n}](\eta)=\big\{x\in\can: \big(\forall i<\omega
  \big) \big(x\rest J_i\in \pi_i(\eta(|H_i|-1)) \big) \big\}$ and\\ 
$\bF[\bar{n}](\eta)=\big\{x\in\can: \big(\forall^\infty i<\omega \big)
\big(x \rest J_i\in \pi_i(\eta(|H_i|-1)) \big) \big\}.$ 
\end{enumerate}
\end{definition}

\begin{lemma}
\label{itspos}
For every $\eta\in\prod\limits_{m<\omega} (m+1)$, $\bF_0[\bar{n}](\eta)
\subseteq \can$ is a closed set of positive Lebesgue measure, and
$\bF[\bar{n}](\eta)$ is a $\Sigma^0_2$ set of Lebesgue measure 1.
\end{lemma}

\begin{proof}
  Note that $J_i\cap J_j=\emptyset$ and $|H_i|<|H_j|$ for $i<j$, and
  $\sum\limits_{i=0}^\infty 2^{-N(2^i)}<1$.  
\end{proof}

\begin{lemma}
\label{combheart}
Let $i<\omega$, $c\in\{0,2\}$ and let $\eta\in {}^{J_i} 2$. Suppose that for
each $\ell<2^i$ and $x<2$ we are given a function $\cZ_\ell^x:
H_i\longrightarrow {}^{J_i} 2$ such that $\cZ_\ell^x(a)\in a$ for each $a\in
H_i$. Then there are $a^0,a^1\in H_i$ such that for every $\ell<2^i$ there
is $k\in [m_{2^i+\ell},m_{2^i+\ell+1})$ satisfying
\[\big( \cZ_\ell^0(a^0)\rest [n_k,n_{k+1}) \big)
\circledast^k_c \big (\cZ_\ell^1(a^1)\rest [n_k,n_{k+1}) \big) = \eta\rest
[n_k,n_{k+1}) ,\]
where $\circledast^k_c$ denotes the operation $\circledast_c$ on
${}^{[n_k,n_{k+1})}2$.   
\end{lemma}

\begin{proof}
We start the proof with the following Claim.

\begin{claim}
\label{cl1}
If $\cA\subseteq H_i$, $|\cA|\leq 2^{|J_i|-N(2^i)-i}$ and $x<2$, then there
is $b\in H_i$ such that $\cZ_\ell^x(b)\notin \{\cZ_\ell^x(a): a\in\cA\}$ for
each $\ell<2^i$. 
\end{claim}

\begin{proof}[Proof of the Claim]
  Note that $|\{\cZ_\ell^x(a):\ell<2^i\ \&\ a\in\cA\}|\leq 2^i\cdot 
  2^{|J_i|-N(2^i)-i}=2^{|J_i|-N(2^i)}$, so letting $b={}^{J_i}2\setminus
  \{\cZ_\ell^x(a): \ell<2^i\ \&\ a\in \cA\}$ we have $b\in H_i$. Since 
  $\cZ_\ell^x(b)\in b$ we see that $b$ is as required in the claim.   
\end{proof}

It follows from Claim \ref{cl1} that we may pick sequences $\langle a_j^0: 
j<j^*\rangle\subseteq H_i$ and $\langle a_j^1:j<j^*\rangle\subseteq H_i$ 
with $\cZ_\ell^x(a^x_{j_1})\neq \cZ_\ell^x(a^x_{j_2})$ for $j_1<j_2<j^*$, 
$\ell<2^i$, $x<2$ and such that $j^*>2^{|J_i|-N(2^i)-i}$. Now,
by induction on $\ell<2^i$, we choose sets $X_\ell,Y_\ell\subseteq j^*$ and
integers $k_\ell\in [m_{2^i+\ell},m_{2^i+\ell+1})$ such that the following
demands are satisfied.
\begin{enumerate}
\item[(i)] $X_{\ell+1}\subseteq X_\ell\subseteq j^*$, $Y_{\ell+1} \subseteq
  Y_\ell \subseteq j^*$,
\item[(ii)] if $j_0\in X_\ell$ and $j_1\in Y_\ell$ then 
\[\big(\cZ_\ell^0(a_{j_0}^0)\rest [n_{k_\ell},n_{k_\ell+1})\big)
\circledast^{k_\ell}_c \big(\cZ_\ell^1(a_{j_1}^1)\rest
[n_{k_\ell},n_{k_\ell+1})\big) =\eta\rest  [n_{k_\ell},n_{k_\ell+1}),\] 
\item[(iii)] $\min\big(|X_\ell|,|Y_\ell|\big)\geq j^*\cdot 2^{N(2^i) -
    N(2^i+\ell+1)-\ell-1}$. 
\end{enumerate}
We stipulate $X_{-1}=Y_{-1}=j^*$ and we assume that $X_{\ell-1},Y_{\ell-1}$
have been already determined (and $\min\big(|X_{\ell-1}|,|Y_{\ell-1}|\big)
\geq j^*\cdot 2^{N(2^i)-N(2^i+\ell)-\ell}$ if $\ell>0$). Let 
\[\begin{array}{l}
X^*=\big\{j\in X_{\ell-1}: |X_{\ell-1}|\cdot 2^{N(2^i+\ell)-N(2^i+
  \ell+1)-1} \leq\\
\ \big|\{j'\in X_{\ell-1}:\cZ_\ell^0(a_{j'}^0)\rest [N(2^i{+} 
\ell), N(2^i{+}\ell{+}1)) = \cZ_\ell^0(a_j^0)\rest
[N(2^i{+}\ell),N(2^i{+}\ell{+}1)) \} \big|\big\},\\
Y^*=\big\{j\in Y_{\ell-1}: |Y_{\ell-1}|\cdot 2^{N(2^i+\ell)-N(2^i+
  \ell+1)-1} \leq\\
\ \big|\{j'\in Y_{\ell-1}:\cZ_\ell^1(a_{j'}^1)\rest [N(2^i{+}
\ell), N(2^i{+}\ell{+}1)) = \cZ_\ell^1(a_j^1)\rest
[N(2^i{+}\ell),N(2^i{+}\ell{+}1)) \} \big|\big\}. 
\end{array}\] 

\begin{claim}
\label{cl2}
$|X^*|\geq \frac{1}{2} |X_{\ell-1}|$ and  $|Y^*|\geq \frac{1}{2} |Y_{\ell-1}|$.
\end{claim}

\begin{proof}[Proof of the Claim]
Assume towards contradiction that $|X^*|<\frac{1}{2}|X_{\ell-1}|$. Then for
some $\nu_0\in {}^{[N(2^i+\ell),N(2^i+\ell+1))} 2$ we have 
\[\begin{array}{r}
\big|\big\{j\in X_{\ell-1}\setminus X^*: \nu_0\subseteq\cZ_\ell^0(a_j^0) \big\}
\big| \geq |X_{\ell-1}\setminus X^*|\cdot 2^{N(2^i+\ell)-N(2^i+\ell+1)}> \\
\frac{1}{2}|X_{\ell-1}|\cdot 2^{N(2^i+\ell)-N(2^i+\ell+1)}.
\end{array}\]
Let $j\in X_{\ell-1}\setminus X^*$ be such that $\nu_0\subseteq
\cZ_\ell^0(a_j^0)$. Then $j\in X^*$, a contradiction.

Similarly for $Y^*$. 
\end{proof}

\begin{claim}
\label{cl3}
For some $k\in [m_{2^i+\ell},m_{2^i+\ell+1})$  we have that both 
$\big|\big\{ \cZ_\ell^0(a_j^0)\rest [n_k,n_{k+1}):j\in X^*\big\}\big|> 
2^{n_{k+1}-n_k-1}$  and $\big|\big\{ \cZ_\ell^1(a_j^1)\rest
[n_k,n_{k+1}):j\in Y^*\big\}\big|> 2^{n_{k+1}-n_k-1}$.
\end{claim}

\begin{proof}[Proof of the Claim]
Let 
\[K^X=\big\{k\in [m_{2^i+\ell},m_{2^i+\ell+1}):   |\{ \cZ_\ell^0(a_j^0)\rest  
[n_k,n_{k+1}):j\in X^*\}|\leq 2^{n_{k+1}-n_k-1}\big\}\]
and  
\[K^Y=\big\{k\in [m_{2^i+\ell},m_{2^i+\ell+1}):   |\{ \cZ_\ell^1(a_j^1)\rest 
[n_k,n_{k+1}):j\in Y^*\}|\leq 2^{n_{k+1}-n_k-1}\big\}.\]
Assume towards contradiction that $|K^X|\geq \frac{1}{2}(m_{2^i+\ell+1}
-m_{2^i+\ell})$. Then 
\[|X^*|=|\{\cZ_\ell^0(a_j^0):j\in X^*\}|\leq 2^{-1/2 (m_{2^i+\ell+1}-
  m_{2^i+\ell})} \cdot 2^{|J_i|}< 2^{|J_i|}\cdot 2^{-4N(2^i+\ell)}.\] 
(Remember \ref{translation}$(*)_1$.)  Hence $|X_{\ell-1}|\leq 2^{|J_i|- 
  4N(2^i+\ell)+1}$. If $\ell=0$ then we get $2^{|J_i|-2N(2^i)}< j^*\leq
2^{|J_i|- 4N(2^i)+1}$, which is impossible.  If $\ell>0$, then by the
inductive hypothesis (iii) we know that $|X_{\ell-1}|\geq j^*\cdot
2^{N(2^i)-N(2^i+\ell)-\ell}>2^{|J_i|-i-N(2^i+\ell)-\ell}$, so
$3N(2^i+\ell)-1<i+\ell$, a clear contradiction. Consequently  $|K^X|<
\frac{1}{2}(m_{2^i+\ell+1} -m_{2^i+\ell})$, and similarly $|K^Y|<
\frac{1}{2}(m_{2^i+\ell+1} -m_{2^i+\ell})$. Pick $k\in
[m_{2^i+\ell}, m_{2^i+\ell+1})$ such that $k\notin K^X\cup K^Y$. 
\end{proof}
Now, let $k_\ell\in [m_{2^i+\ell}, m_{2^i+\ell+1})$ be as given by Claim
\ref{cl3}. Necessarily the sets 
$\big\{\rho\in {}^{[n_{k_\ell},n_{k_\ell+1})} 2: (\exists j\in X^*)(
(\cZ_\ell^0(a_j^0)\rest [n_{k_\ell},n_{k_\ell+1}))\circledast_c^{k_\ell} \rho
= \eta\rest [n_{k_\ell},n_{k_\ell+1}))\big\}$ and
$\big\{\cZ_\ell^1(a_j^1)\rest [n_{k_\ell},n_{k_\ell+1}):j\in Y^*\big\}$
have non-empty intersection. Therefore, we may find $j_X\in X^*$ and $j_Y\in
Y^*$ such that   
\[\big(\cZ_\ell^0(a_{j_X}^0)\rest [n_{k_\ell},n_{k_\ell+1})\big)
\circledast_c^{k_\ell} \big(\cZ_\ell^1(a_{j_Y}^1)\rest
[n_{k_\ell},n_{k_\ell+1})\big) =  \eta\rest [n_{k_\ell},n_{k_\ell+1}).\]  
Set 
\[X_\ell=\big\{j\in X_{\ell-1}: \cZ_\ell^0(a_j^0)\rest [N(2^i+\ell),
N(2^i+\ell+1))= \cZ_\ell^0(a_{j_X}^0)\rest [N(2^i+\ell),
N(2^i+\ell+1))\big\},\]
and 
\[Y_\ell=\big\{j\in Y_{\ell-1}: \cZ_\ell^1(a_j^1)\rest [N(2^i+\ell),
N(2^i+\ell+1))= \cZ_\ell^1(a_{j_Y}^1)\rest [N(2^i+\ell),
N(2^i+\ell+1))\big\}.\]
By the definition of $X^*,Y^*$ and by the inductive hypothesis (iii) we have  
\[|X_\ell|\geq |X_{\ell-1}|\cdot 2^{N(2^i+\ell)-N(2^i+\ell+1) -1} \geq
j^*\cdot 2^{N(2^i)-\ell -N(2^i+\ell+1) -1}\] 
and similarly for $Y_\ell$. Consequently, $X_\ell,Y_\ell$ and $k_\ell$
satisfy the inductive demands (i)--(iii).

After the above construction is completed fix any $j_0\in X_{2^i-1}$,
$j_1\in Y_{2^i-1}$ and consider $a^0=a_{j_0}$ and $a^1=a_{j_1}$. For each 
$\ell<2^i$ we have $j_0\in X_\ell$, $j_1\in Y_\ell$ so 
\[\big(\cZ_\ell^0(a^0)\rest [n_{k_\ell},n_{k_\ell+1})\big)
\circledast^{k_\ell}_c \big(\cZ_\ell^1(a^1)\rest
[n_{k_\ell},n_{k_\ell+1})\big) = \eta\rest [n_{k_\ell},n_{k_\ell+1}).\]  
Hence $a^1,a^2\in H_i$ are as required.  
\end{proof}

\subsection{The $*$--Silver forcing notion}
The consistency result of the next section will be obtained using CS
product of the following forcing notion $\bbS_*$. 

\begin{definition}
  \label{fordef}
  \begin{enumerate}
  \item We define the $*$--Silver forcing notion $\bbS_*$ as follows. \\
{\bf A condition } in $\bbS_*$ is a partial function $p:\dom(p)
\longrightarrow \omega$ such that $\dom(p)\subseteq\omega$ is coinfinite and
$p(m)\leq m$ for each $m\in\dom(p)$.\\
{\bf The order } $\leq=\leq_{\bbS_*}$ of $\bbS_*$ is the inclusion, i.e.,
$p\leq q$ if and only if $p\subseteq q$. 
  \item  For $p\in \bbS_*$ and $1\leq n<\omega$ we let $u(n,p)$ be the 
    set of the first $n$ elements of $\omega\setminus \dom(p)$ (in the
    natural increasing order).  Then for $p,q\in\bbS_*$ we let\\
$p\leq_n q$ if and only if $p\leq q$ and $u(n,q)=u(n,p)$.

We also define $p\leq_0 q$ as equivalent to $p\leq q$.
\item Let $p\in\bbS_*$. We let $S(n,p)$ be the set of all functions $s:
  u(n,p)\longrightarrow \omega$ with the property that $s(m)\leq m$ for all
  $m\in u(n,p)$.
\item We let $\name{\eta}$ to be the canonical $\bbS_*$--name such that  
\[\forces \name{\eta}=\bigcup\{p:p\in\name{G}_{\bbS_*}\}.\]
  \end{enumerate}
\end{definition}

\begin{remark}
  The forcing notion $\bbS_*$ may be represented as a forcing of the type
  $\bbQ_{{\rm w}\infty}^*(K,\Sigma)$ for some finitary creating pair 
  $(K,\Sigma)$ which captures singletons, see Ros{\l}anowski and Shelah
  \cite[Definition 2.1.10]{RoSh:470}. It is a close relative of the Silver
  forcing notion and, in a sense, it lies right above all $\bbS_n$'s studied
  for instance in Ros{\l}anowski \cite{Ro0x} and Ros{\l}anowski and
  Stepr\={a}ns  \cite{RoSt08}.   
\end{remark}

\begin{lemma}
  \label{basiclemma}
  \begin{enumerate}
\item $(\bbS_*,\leq_{\bbS_*})$ is a partial order of size $\con$. If
  $p\in\bbS_*$ and $s\in S(n,p)$ then $p\cup s\in\bbS_*$ is a condition
  stronger than $p$.  
\item $\forces_{\bbS_*} \name{\eta}\in \prod\limits_{m<\omega} (m+1)$ and
  $p\forces_{\bbS_*} p\subseteq\name{\eta}$ (for $p\in\bbS_*$). 
\item If $p\in\bbS_*$ and $1\leq n<\omega$, then the family $\{p\cup s:s\in
  S(n,T)\}$ is an antichain pre-dense above $p$.  
\item The relations $\leq_n$ are partial orders on $\bbS_*$, $p\leq_{n+1} q$
  implies $p\leq_n q$.
\item Assume that $\name{\tau}$ is an $\bbS_*$--name for an ordinal,
  $p\in\bbS_*$, $1\leq n,m<\omega$. Then there is a condition $q\in\bbS_*$
  such that $p\leq_n q$, $\max\big(u(n+1,q)\big)>m$ and for all $s\in
  S(n,q)$ the condition $q\cup s$ decides the value of $\name{\tau}$.
\item The forcing notion $\bbS_*$ satisfies Axiom A of Baumgartner \cite[\S
  7]{B3} as witnessed by the orders $\leq_n$, it is $\baire$--bounding and,
  moreover, every meager subset of $\can$ in an extension by $\bbS_*$ is
  included in a $\Sigma^0_2$ meager set coded in the ground model. 
  \end{enumerate}
\end{lemma}

\begin{proof}
Straightforward - the same as for the Silver forcing notion.
\end{proof}

\begin{definition}
\label{product}
Assume $\kappa^{\aleph_0} = \kappa\geq\aleph_2$.
\begin{enumerate}
\item  $\bbS_*(\kappa)$ is the CS product of $\kappa$ many copies of
  $\bbS_*$. Thus\\
{\bf a condition} $p$ in $\bbS_*(\kappa)$ is a function with a countable 
domain $\dom(p)\subseteq \kappa$ and with values in $\bbS_*$, and\\ 
{\bf the order} $\leq$ of $\bbS_*(\kappa)$ is such that\\ 
$p\leq q$ if and only if $\dom(p)\subseteq \dom(q)$ and $(\forall \alpha\in
\dom(p))(p(\alpha)\leq_{\bbS_*} q(\alpha))$. 
\item Suppose that $p\in\bbS_*(\kappa)$ and $F\subseteq \dom(p)$ is a finite
  non-empty set and $\mu:F\longrightarrow\omega\setminus\{0\}$. Let
  $v(F,\mu,p)=\prod\limits_{\alpha\in F} u(\mu(\alpha),p(\alpha))$ and 
  $T(F,\mu,p)=\prod\limits_{\alpha\in F} S(\mu(\alpha),p(\alpha))$. 

  If $\sigma \in T(F,\mu,p)$ then let $p|\sigma$ be the condition
  $q\in\bbS_*(\kappa)$ such that $\dom(q)=\dom(p)$ and
  $q(\alpha)=p(\alpha)\cup\sigma(\alpha)$ for $\alpha\in F$ and
  $q(\alpha)=p(\alpha)$ for $\alpha\in \dom(q)\setminus F$.  

We let $p\leq_{F,\mu} q$ if and only if $p\leq q$ and $v(F,\mu,p)=
v(F,\mu,q)$.   

If $\mu$ is constantly $n$ then we may write $n$ instead of $\mu$.

\item Suppose that $p\in\bbS_*(\kappa)$ and $\name{\bar{\tau}}= \langle 
  \name{\tau}_n: n<\omega\rangle$ is a sequence of names for ordinals. We
  say that {\em $p$ determines $\name{\bar{\tau}}$ relative to $\bar{F}$} if 
  \begin{itemize}
\item $\bar{F}=\langle F_n: n<\omega\rangle$ is a sequence of finite subsets
  of $\dom(p)$, and  
\item $p$ forces a value to $\name{\tau}_0$ and for $1\leq n<\omega$ and
  $\sigma\in T(F_n,n,p)$ the condition $p|\sigma$ decides the value of
  $\name{\tau}_n$.   
  \end{itemize}
\end{enumerate}
\end{definition}

\begin{lemma}
 \label{blprod}
 \begin{enumerate}
\item The forcing notion $\bbS_*(\kappa)$ satisfies $\con^+$--chain
  condition.
\item Suppose that $p\in\bbS_*(\kappa)$, $F\subseteq \dom(p)$ is finite
  non-empty, $\mu:F\longrightarrow \omega\setminus\{0\}$ and $\name{\tau}$
  is a name for an ordinal. Then there is a condition $q\in\bbS_*(\kappa)$ such that
 $p\leq_{F,\mu} q$ and for every $\sigma\in T(F,\mu,q)$ the condition
 $q|\sigma$ decides the value of $\name{\tau}$.  
\item Suppose that $p\in\bbS_*(\kappa)$ and $\name{\bar{\tau}}= \langle 
  \name{\tau}_n: n<\omega\rangle$ is a sequence of $\bbS_*(\kappa)$--names
  for objects from the ground model $\bV$. Then there is a condition $q\geq
  p$ and a $\subseteq$--increasing sequence $\bar{F}=\langle
  F_n:n<\omega\rangle$ of finite subsets of $\dom(q)$ such that $q$
  determines $\name{\bar{\tau}}$ relative to $\bar{F}$.  
\item Assume $p,\name{\bar{\tau}}$ are as in (3) above and $p\forces$
  ``$\name{\bar{\tau}}$ is a sequence of elements of $\pfs$ with
  disjoint domains''.  Then there are a condition $q\geq p$ and an increasing
  sequence $\bar{F}$ of finite subsets of $\dom(q)$ and a function
  $f=(f_0,f_1):\bigcup\limits_{1\leq n<\omega} T(F_n,n,q) \longrightarrow
  \omega\times\pfs$ such that $q|\sigma\forces\name{\tau}_{f_0(\sigma)} = 
  f_1(\sigma)$ (for all $\sigma\in\dom(f)$) and the elements of
$\langle \dom(f_1(\sigma)):\sigma\in \bigcup_{n<\omega} T(F_n,n,q) \rangle$
are pairwise disjoint. 
 \end{enumerate}
\end{lemma}

\begin{proof}
The same as for the CS product of Silver or Sacks forcing notions, see
e.g.~Baumgartner \cite[\S 1]{Ba85}. 
\end{proof}

\begin{corollary}
\label{corprod}
Assume $\kappa=\kappa^{\aleph_0}\geq \aleph_2$. The forcing notion
$\bbS_*(\kappa)$ is proper and every meager subset of $\can$ in an extension
by $\bbS_*(\kappa)$ is included in a $\Sigma^0_2$ meager set coded in the
ground model.   

\noindent If CH holds, then $\bbS_*(\kappa)$ preserves all cardinals and
cofinalities and $\forces_{\bbS_*(\kappa)} 2^{\aleph_0} =\kappa$.
\end{corollary}

\section{Meager non--null}
The goal of this section is to present a model of ZFC in which every 
meager subgroup of $\mbR$ or $\can$ is also Lebesgue null. 

\begin{theorem}
\label{mainthm}
Assume CH. Let $\kappa=\kappa^{\aleph_0}\geq\aleph_2$. Then 
\begin{enumerate}
\item $\forces_{\bbS_*(\kappa)}$`` $2^{\aleph_0}=\kappa$ and  every 
  meager subgroup of $(\can,+_2)$ is Lebesgue null. ''
\item $\forces_{\bbS_*(\kappa)}$`` every meager subgroup of $(\mbR,+)$
  is Lebesgue null. '' 
\end{enumerate}
\end{theorem}

\begin{proof}
For $\alpha<\kappa$ let $\name{\eta}_\alpha$ be the canonical  
name for the $\bbS_*$--generic function in $\prod\limits_{m<\omega} (m+1)$
added on the $\alpha^{\rm th}$ coordinate of $\bbS_*(\kappa)$. 

\noindent (1)\qquad Suppose towards contradiction that for some
$p_0\in\bbS_*(\kappa)$ and a  $\bbS_*(\kappa)$--name $\name{H}$ we have   
\[p_0\forces_{\bbS_*(\kappa)} \mbox{`` $\name{H}$ is a meager non--null
  subgroup of $(\can,+_2)$. ''}\]  
By Corollary \ref{corprod} (or, actually, Lemma \ref{blprod}(4)) we may pick
a condition $p_1\geq p_0$, a strictly increasing sequence $\bar{n}=\langle 
n_j: j<\omega\rangle\subseteq \omega$ and a function $f\in\can$ such that 
\begin{enumerate}
\item[$(*)_0$] $p_1\forces_{\bbS_*(\kappa)} \mbox{`` }\name{H}\subseteq
  \big\{ x\in \can: \big (\forall^\infty j<\omega\big)\big(x\rest
  [n_j,n_{j+1})\neq f\rest [n_j,n_{j+1}) \big) \big\}  \mbox{. ''}$
\end{enumerate}
Let $\bar{m}=\bar{m}[\bar{n}]$, $\bar{N}=\bar{N}[\bar{n}]$,
$\bar{J}=\bar{J}[\bar{n}]$, $\bar{H}=\bar{H}[\bar{n}]$, $\pi=\pi[\bar{n}]$
and $\bF=\bF[\bar{n}]$ be as defined in Definition \ref{translation} for
the sequence $\bar{n}$.  Also let $A=\{|H(i)|-1:i<\omega\}$ and 
$r^+\in\bbS_*$ be such that $\dom(r^+)=\omega\setminus A$ and $r^+(k)=0$ for 
$k\in \dom(r^+)$.

Since, by Lemma \ref{itspos}, we have $\forces$``
$\bF(\name{\eta}_\alpha)\subseteq \can$ is a measure one set '', we know
that $p_1\forces_{\bbS_*(\kappa)} \mbox{`` }(\forall\alpha<
\kappa)(\bF(\name{\eta}_\alpha) \cap\name{H}\neq\emptyset )\mbox{
  ''}$. Consequently, for each $\alpha<\kappa$, we may choose a
$\bbS_*(\kappa)$--name $\name{\rho}_\alpha$ for an element of $\can$ such
that  
\[p_1\forces_{\bbS_*(\kappa)}\mbox{`` }\name{\rho}_\alpha\in \name{H}
\ \&\ \name{\rho}_\alpha\in \bF(\name{\eta}_\alpha) \mbox{ ''.}\]
Let us fix $\alpha\in\kappa\setminus \dom(p_1)$ for a moment. Let
$p_1^\alpha\in\bbS_*(\kappa)$ be  a condition such that $\dom(p_1^\alpha)=
\dom(p_1)\cup\{\alpha\}$, $p_1^\alpha(\alpha)=r^+$ and $p_1\subseteq
p_1^\alpha$. Using the standard fusion based argument (like the one applied
in the classical proof of Lemma \ref{blprod}(3) with \ref{blprod}(2) used
repeatedly), we may find a condition $q^\alpha\in\bbS_*(\kappa)$,  a
sequence $\bar{F}=\langle F^\alpha_n: n<\omega \rangle$ of finite sets, a
sequence $\langle \mu^\alpha_n:n<\omega\rangle$ and an integer
$i^\alpha<\omega$ such that the following demands $(*)_1$--$(*)_6$ are
satisfied. 
\begin{enumerate}
\item[$(*)_1$] $q^\alpha\geq p^\alpha_1$, $\dom(q^\alpha)=
  \bigcup\limits_{n<\omega} F^\alpha_n$, $F^\alpha_n\subseteq
  F^\alpha_{n+1}$ and $F^\alpha_0=\{\alpha\}$.
\item[$(*)_2$] $\mu^\alpha_n:F^\alpha_n\longrightarrow\omega$, $\mu^\alpha_n
  (\alpha)=n+1$, $\mu^\alpha_n(\beta)=n$ for $\beta\in F^\alpha_n\setminus\{
  \alpha\}$. 
\item[$(*)_3$] $\min\big(\omega\setminus \dom(q^\alpha(\alpha))\big)
  >|H(i^\alpha)|$ and\\
if $\max\big(u(n+1,q^\alpha(\alpha))\big)=|H(i)|-1$ and $n\geq 1$, then
$|T(F_n,n,q^\alpha)|^2<2^i$,   
\item[$(*)_4$] $q^\alpha\forces \big(\forall i\geq i^\alpha \big) \big(
  \name{\rho}_\alpha\rest J_i\in \pi_i(\name{\eta}_\alpha(|H_i|-1))\big)$,
  and 
\item[$(*)_5$] $q^\alpha$ determines $\name{\rho}_\alpha$ relative to
  $\bar{F}$, moreover 
\item[$(*)_6$] if $\sigma\in T(F^\alpha_n,\mu^\alpha_n,q^\alpha)$ and
  $\max\big(u(n+1,q^\alpha(\alpha))\big)=|H(i)|-1$, then $q^\alpha|\sigma$
  decides the value of $\name{\rho}_\alpha\rest J_i$.
\end{enumerate}

Unfixing $\alpha$ and using a standard $\Delta$--system argument with CH we
may find distinct $\gamma,\delta\in\kappa\setminus \dom(p_1)$ such that 
${\rm otp}(\dom(q^\gamma))={\rm otp}(\dom(q^\delta))$ and if 
$g:\dom(q^\gamma)\longrightarrow \dom(q^\delta)$ is the order preserving
bijection, then the following demands $(*)_7$--$(*)_9$ hold true. 
\begin{enumerate}
\item[$(*)_7$] $i^\gamma=i^\delta$, $g\rest \big(\dom(q^\gamma)\cap
  \dom(q^\delta)\big)$ is the identity, $g(\gamma)=\delta$,
\item[$(*)_8$] $q^\gamma(\beta)=q^\delta(g(\beta))$ for each $\beta\in
  \dom(q^\gamma)$, and $g[F^\gamma_n]=F^\delta_n$, 
\item[$(*)_9$] if $F\subseteq \dom(q^\delta)$ is finite,
  $\mu:F\longrightarrow \omega\setminus\{0\}$, $i<\omega$, 
  $\sigma\in T(F,\mu,q^\delta)$, then 
\[q^\delta|\sigma\forces \name{\rho}_\delta\rest J_i=z\quad\mbox{ if and
  only if }\quad q^\gamma|(\sigma\circ g)\forces \name{\rho}_\gamma \rest
J_i=z.\]  
\end{enumerate}
Clearly $q^*\stackrel{\rm def}{=}q^\gamma\cup q^\delta$ is a condition
stronger than both $q^\gamma$ and $q^\delta$. Let $F^*_n=F^\gamma_n \cup
F^\delta_n$ for $n<\omega$. 

Let $\langle k_\ell:\ell<\omega\rangle$ be the increasing enumeration of 
$\omega\setminus\dom(q^\gamma(\gamma)) = \omega\setminus \dom(q^\delta( 
\delta))$. Note that by the choice of $r^+$ and $p_1^\gamma$, we have
$\omega\setminus \dom(q^\gamma(\gamma))\subseteq A$, so each $k_\ell$ is of
the form $|H(i)|-1$ for some $i$.  Now we will choose conditions
$r_\delta,r_\gamma\in\bbS_*$ so that  
\[\dom(r_\delta)=\dom(r_\gamma)=\dom(q^\delta(\delta))\cup\{k_{2\ell}:
\ell< \omega\},\]
$q^\delta(\delta)\leq r_\delta$, $q^\gamma(\gamma)\leq r_\gamma$ and  the
values of $r_\delta(k_{2\ell}),r_\gamma(k_{2\ell})$ are picked as follows.   

Let $i$ be such that $k_{2\ell}=|H(i)|-1$. If $x\in\{\gamma,\delta\}$ and 
$\sigma\in T(F^x_{2\ell},\mu^x_{2\ell},q^x)$ then $q^x|\sigma$ decides the
value of $\name{\rho}_x\rest J_i$ (by $(*)_6$) and this value belongs to
$\pi_i\big(\sigma(x)(k_{2\ell})\big)$ (by
$(*)_4+(*)_3$). Consequently, for $x\in\{\gamma,\delta\}$ and 
$\tau\in T(F_{2\ell}^*,2\ell,q^*)$ we may define a function
$\cZ_\tau^x:H(i)\longrightarrow {}^{J_i} 2$ so that  
\begin{enumerate}
\item[$(*)_{10}$] {\bf if} $a\in H(i)$, $\mu:F_{2\ell}^* \longrightarrow \omega$
  is such that $\mu(x)=2\ell+1$ and $\mu(\alpha)=2\ell$ for $\alpha\neq x$,
  and $\tau_a\in  T(F_{2\ell}^*,\mu,q^*)$ is such that
  $\tau_a(\alpha)=\tau(\alpha)$  for $\alpha\in F^*_{2\ell}\setminus \{x\}$
  and $\tau_a(x)=\tau(x)\cup\{(k_{2\ell},a)\}$,\\ 
{\bf then} $q^*|\tau_a\forces_{\bbS_*(\kappa)} \name{\rho}_x\rest
J_i=\cZ^x_\tau(a)$ and $\cZ^x_\tau(a)\in a$. 
\end{enumerate}
Since $|T(F_{2\ell}^*,2\ell,q^*)|\leq
|T(F_{2\ell}^\gamma,2\ell,q^\gamma)|^2< 2^i$ (remember $(*)_3$), we may use
Lemma \ref{combheart} to find $r_\delta(k_{2\ell}),r_\gamma(k_{2\ell})\leq
k_{2\ell}$ such that   
\begin{enumerate}
\item[$(*)_{11}$] for every $\tau\in T(F_{2\ell}^*,2\ell,q^*)$ there is
  $k\in  [m_{2^i},m_{2^{i+1}})$ satisfying 
\[\big( \cZ_\tau^\gamma(\pi_i(r_\gamma(k_{2\ell})))\rest [n_k,n_{k+1}) \big)
+_2 \big (\cZ_\tau^\delta(\pi_i(r_\delta(k_{2\ell})))\rest [n_k,n_{k+1})
\big) = f\rest [n_k,n_{k+1}) .\] 
(Remember, $f$ was chosen in $(*)_0$.)  
\end{enumerate}
This completes the definition of $r_\gamma$ and $r_\delta$. Let
$q^+\in\bbS_*(\kappa)$ be such that $\dom(q^+)=\dom(q^*)= \dom(q^\gamma)\cup
\dom(q^\delta)$ and $q^+(\alpha)=q^*(\alpha)$ for
$\alpha\in\dom(q^+)\setminus\{\gamma,\delta\}$ and $q^+(\gamma)=r_\gamma$
and $q^+(\delta)=r_\delta$. Then $q^+$ is a (well defined) condition
stronger than both $q^\gamma$ and $q^\delta$ and such that 
\begin{enumerate}
\item[$(\clubsuit)$] $q^+\forces \Big(\exists^\infty k<\omega\Big)\Big(
  \big( \name{\rho}_\gamma\rest [n_k,n_{k+1}) \big)+_2 \big
  (\name{\rho}_\delta\rest [n_k,n_{k+1}) \big) = f\rest [n_k,n_{k+1}) \Big)$ 
\end{enumerate}
(by $(*)_{10}+(*)_{11}$). Consequently, by $(*)_0$, 
\begin{enumerate}
\item[$(\heartsuit)$] $q^+\forces$`` $\name{\rho}_\gamma,
  \name{\rho}_\delta\in \name{H}$ and $\name{\rho}_\gamma+_2 \name{\rho}_\delta
  \notin\name{H}$ and $(\name{H},+_2)$ is a group'',
\end{enumerate}
a contradiction.
\medskip

\noindent (2)\qquad The proof is a small modification of that for the first
part, so we describe the new points only. Assume towards contradiction that
for some $p_0\in\bbS_*(\kappa)$ and a  $\bbS_*(\kappa)$--name $\name{H}^*$ we
have   
\[p_0\forces_{\bbS_*(\kappa)} \mbox{`` $\name{H}^*$ is a meager non--null 
  subgroup of $(\mbR,+)$ ''.}\]  
Let $\name{H}_0,\name{H}_1$ be $\bbS_*$--names for subsets of $D^\infty_0$
such that 
\[p_0\forces_{\bbS_*(\kappa)}\mbox{`` }\name{H}_0=\bE^{-1}[\name{H}^* \cap
[0,1/2)] \mbox{ and } \name{H}_1=\bE^{-1}[\name{H}^* \cap [0,1)] \mbox{
  ''.}\]   
Necessarily $p_0\forces$`` $\name{H}^*\cap [0,1/2)$ is not null '', so it
follows from \ref{expanding}(1) that  
\[p_0\forces_{\bbS_*(\kappa)}\mbox{`` }\name{H}_0\notin \cN\mbox{ and }
\name{H}_1\in \cM\mbox{ and }\name{H}_0\subseteq \name{H}_1 \mbox{
  ''.}\]   
Clearly we may pick a condition $p_1\geq p_0$, a sequence $\bar{n}=\langle   
n_j: j<\omega\rangle\subseteq \omega$ and a function $f\in\can$ such that 
\begin{enumerate}
\item[$(\oplus)_0$] $n_{j+1}>n_j+j+1$ for each $j$,
\item[$(\oplus)_1$] $f(n_{j+1}-1)=0$ for each $j$, and 
\item[$(\oplus)_2$] $p_1\forces_{\bbS_*(\kappa)}$``$\name{H}_1\subseteq 
  \big\{ x\in \can: \big (\forall^\infty j<\omega\big)\big(x\rest
  [n_j,n_{j+1}{-}1)\neq f\rest [n_j,n_{j+1}{-}1) \big) \big\}$.''\\
(Note: ``$ [n_j,n_{j+1}-1)$'' not ``$ [n_j,n_{j+1})$''.)
\end{enumerate}
Like in part (1), let $\bar{m}=\bar{m}[\bar{n}]$,
$\bar{N}=\bar{N}[\bar{n}]$, $\bar{J}=\bar{J}[\bar{n}]$,
$\bar{H}=\bar{H}[\bar{n}]$, $\pi=\pi[\bar{n}]$ and $\bF=\bF[\bar{n}]$.  Let
$A=\{|H(i)|-1:i<\omega\}$ and  $r^+\in\bbS_*$ be such that
$\dom(r^+)=\omega\setminus A$ and $r^+(k)=0$ for  $k\in \dom(r^+)$. 
Then each $\alpha<\kappa$ fix a $\bbS_*(\kappa)$--name
$\name{\rho}_\alpha$ such that $p_1\forces_{\bbS_*(\kappa)}$``
$\name{\rho}_\alpha\in \name{H}_0\cap \bF(\name{\eta}_\alpha)$ ''. 

Now repeat the arguments of the first part (with $(*)_1$--$(*)_{11}$ there
applied to our $\bar{n},f,\name{\rho}_\alpha$ and $\circledast_0$ here) to
find $\gamma,\delta\in \kappa\dom(p_1)$ such that 
\begin{enumerate}
\item[$(\lozenge)$] $q^+\forces$`` $\big(\exists^\infty k<\omega\big)\big( 
  (\name{\rho}_\gamma\rest [n_k,n_{k+1})) \circledast_0
  (\name{\rho}_\delta\rest [n_k,n_{k+1})) =f\rest [n_k,n_{k+1})\big)$ ''.
\end{enumerate}
Let $G\subseteq \bbS_*(\kappa)$ be a generic over $\bV$ such that $q^+\in
G$ and let us work in $\bV[G]$. Let $\eta\in D^\infty_0$ be such that
$\bE(\name{\rho}_\gamma^G) +\bE(\name{\rho}_\delta^G) =\bE(\eta)$ (remember
$\bE(\name{\rho}_\gamma^G),\bE(\name{\rho}_\delta^G)<1/2$). We know from
$(\lozenge)$ that  there are infinitely many $k<\omega$ satisfying 
\begin{enumerate}
\item[$(\blacklozenge)$]  $(\name{\rho}_\gamma^G\rest [n_k,n_{k+1})) \circledast_0 
  (\name{\rho}_\delta^G\rest [n_k,n_{k+1})) =f\rest [n_k,n_{k+1})$. 
\end{enumerate}
Since $f(n_{k+1}-1)=0$ (see $(\oplus)_1$), we get from \ref{carrying}(3)
that for each $k$ as in $(\blacklozenge)$ we also have  
\[\begin{array}{l}
(\name{\rho}_\gamma^G\rest [n_k,n_{k+1}-1))
  \circledast_0  (\name{\rho}_\delta^G\rest [n_k,n_{k+1}-1))=\\
  (\name{\rho}_\gamma^G\rest [n_k,n_{k+1}-1)) \circledast_1
  (\name{\rho}_\delta^G\rest [n_k,n_{k+1}-1)) =f\rest [n_k,n_{k+1}-1).
\end{array}\] 
Therefore (by \ref{carrying}(4)) for each $k$ satisfying $(\blacklozenge)$
we have $\eta\rest [n_k,n_{k+1}-1)=f \rest [n_k,n_{k+1}-1)$, so 
\[\big(\exists^\infty k<\omega\big)\big(\eta\rest [n_k,n_{k+1}-1)=f \rest
[n_k,n_{k+1}-1)\big).\]
Consequently, by $(\oplus)_2$, we have that $\eta\notin\name{H}_1^G$, i.e.,
$\bE(\eta)\notin (\name{H}^*)^G\cap [0,1)$. This contradicts the fact that 
$\bE(\name{\rho}_\gamma^G),\bE(\name{\rho}_\delta^G) \in (\name{H}^*)^G$,
$\bE(\eta)=\bE(\name{\rho}_\gamma^G)+\bE(\name{\rho}_\delta^G)$ and
$(\name{H}^*)^G$ is a subgroup of $(\mbR,+)$. 
\end{proof}

\begin{remark}
Instead of the CS product of forcing notions $\bbS_*$ we could have used
their CS iteration of length $\omega_2$. Of course, that would restrict the
value of the continuum in the resulting model.
\end{remark}

\section{Problems}
Both theorems \ref{firstres}(1) and \ref{mainthm}(1) can be repeated for
other product groups. We may consider a sequence $\langle
H_n:n<\omega\rangle$ of finite groups and their coordinate-wise
product $H=\prod\limits_{n<\omega} H_n$. Naturally, $H$ is equipped with
product topology of discrete $H_n$'s and the product probability
measure. Then there exists a null non--meager subgroup of $H$ but it is
consistent that there is no meager non--null such subgroup. It is natural to
ask now:

\begin{problem}
\begin{enumerate}
\item Does every locally compact group (with complete Haar measure) admit a
  null non--meager subgroup?   
\item Is it consistent that no locally compact group has a meager non--null
  subgroup? 
\end{enumerate} 
\end{problem}

In relation to Theorem \ref{mainthm}, we still should ask:

  \begin{problem}
    Is it consistent that there exists a translation invariant Borel 
    hull for the meager ideal on $\can$? On ${\mbR}$?
  \end{problem}


\begin{thebibliography}{1}

\bibitem{BaJu95}
Tomek Bartoszy\'nski and Haim Judah.
\newblock {\em {Set Theory: On the Structure of the Real Line}}.
\newblock A K Peters, Wellesley, Massachusetts, 1995.

\bibitem{B3}
James~E. Baumgartner.
\newblock {Iterated forcing}.
\newblock In A.~Mathias, editor, {\em Surveys in Set Theory}, volume~87 of {\em
  London Mathematical Society Lecture Notes}, pages 1--59, Cambridge, Britain,
  1978.

\bibitem{Ba85}
James~E. Baumgartner.
\newblock {Sacks forcing and the total failure of Martin's axiom}.
\newblock {\em Topology and its Applications}, 19:211--225, 1985.

\bibitem{FRSh:1031}
Tomasz Filipczak, Andrzej Roslanowski, and Saharon Shelah.
\newblock {On Borel hull operations}.
\newblock {\em Real Analysis Exchange}, 40:129--140, 2015.
\newblock arxiv:1308.3749.

\bibitem{J}
Thomas Jech.
\newblock {\em {Set theory}}.
\newblock Springer Monographs in Mathematics. Springer-Verlag, Berlin, 2003.
\newblock The third millennium edition, revised and expanded.

\bibitem{Ro0x}
Andrzej Ros{\l}anowski.
\newblock {$n$--localization property}.
\newblock {\em Journal of Symbolic Logic}, 71:881--902, 2006.
\newblock arxiv:math.LO/0507519.

\bibitem{RoSh:470}
Andrzej Roslanowski and Saharon Shelah.
\newblock {Norms on possibilities I: forcing with trees and creatures}.
\newblock {\em {Memoirs of the American Mathematical Society}}, 141(671):xii +
  167, 1999.
\newblock arxiv:math.LO/9807172.

\bibitem{RoSt08}
Andrzej Ros{\l}anowski and Juris Stepr\=ans.
\newblock {Chasing Silver}.
\newblock {\em Canadian Mathematical Bulletin}, 51:593--603, 2008.
\newblock arxiv:math.LO/0509392.

\end{thebibliography}

\end{document}